\documentclass[12pt]{amsart}
\usepackage{amsmath,amssymb,enumerate}
\usepackage{latexsym}
\usepackage{amsfonts}
\usepackage {amsbsy}
\usepackage {amsmath}
\usepackage{amssymb}
\usepackage{enumerate}
\usepackage{ bbm}
\newtheorem{theorem}{Theorem}
\newtheorem{lemma}[theorem]{Lemma}
\newtheorem{corollary}[theorem]{Corollary}
\newtheorem{proposition}[theorem]{Proposition}

\theoremstyle{definition}
\newtheorem{definition}[theorem]{Definition}

\numberwithin{theorem}{section}
\numberwithin{equation}{section}

\newcommand{\B}{\mathbb{B}}
\newcommand{\N}{\mathbb{N}}
\newcommand{\R}{\mathbb{R}}
\newcommand{\C}{\mathbb{C}}

\makeatletter
\@namedef{subjclassname@2010}{%
  \textup{2010} Mathematics Subject Classification}
\makeatother

\address{Ahmed Zeriahi, Institut de Math\'ematiques de Toulouse, 
Universit\'e de Toulouse, 
CNRS, UPS, 118 route de Narbonne, 
31062 Toulouse cedex 09, France}
\email{ahmed.zeriahi@math.univ-toulouse.fr}

\begin{document}

\thanks{The  author was partially supported by the ANR project GRACK}
  
%
%


\title[Remarks on the modulus of continuity]{Remarks on the modulus of continuity \\ of subharmonic functions }
\author{Ahmed Zeriahi}

\date{\today}


\maketitle
\begin{abstract} We introduce different classical characteristics used to regularize  a subharmonic function and compare them. 

As an application we give  a complete proof of a useful characterization of the modulus of continuity of such functions in terms of these characteristics under a technical condition. This result is extended to quasi-plurisubharmonic functions on a compact Hermitian manifold.

\end{abstract}

\section{Introduction}

Given a subharmonic function $u$ on a domain $\Omega \subset \R^n$, we introduce various continuous approximating functions of  $u$  and use them to define various "partial moduli  of continuity" associated to $u$. These moduli have been used in many papers to measure in different ways the continuity of solutions to complex Monge-Amp\`ere equations on bounded domains in $\C^n$ as well as on compact complex manifolds (see \cite{GKZ08}, \cite{DDGKPZ14}, \cite{N18}, \cite{KN20}, \cite{BZ20}) .
 
The goal of this note is to clarify the relations between these moduli and establish estimates on the (full) modulus of continuity of a subharmonic function in terms of these partial moduli of continuity.

Let $\Omega \subset \R^n$ be a domain and  $u : \Omega \longrightarrow \R \cup \{- \infty \}$ be a subharmonic function on $\Omega$.

We fix $\delta_0 > 0$ so that $\Omega_{\delta_0} := \{ x \in \Omega \, ; \, \mathrm{dist} (x, \partial \Omega) > \delta_0\} \neq \emptyset$.

For $0 < \delta < \delta_ 0$, we can define the $\delta$-mean value function associated to  $u$ as follows
\begin{equation}
\Lambda_{\delta} u (x) := \int_{\B} u (x + \delta \xi) d \lambda_\B (\xi),
\end{equation}
for $x \in \Omega_\delta := \{ x \in \Omega \, ; \, \mathrm{dist} (x, \partial \Omega) > \delta\}$, where   $\lambda_\B $ is the normalized Lebesgue measure on the euclidean unit ball $\B \subset \R^n$.

We can also consider the $\delta$-max regularization of $u$ defined as follows:
\begin{equation}
\mathcal{M}_\delta u(x) := \max_{y \in \B} u(x + \delta y), 
\end{equation}
for $x \in \Omega_\delta $ and $0 < \delta < \delta_0$.

We have obviously $\Lambda_\delta u (x) \leq \mathcal{M}_\delta u ()$ for any $x \in \Omega_\delta$. Moreover these functions are continuous and subharmonic  on $\Omega_\delta$ and decrease to $u$ pointwise  on $\Omega$.

We consider a modulus of continuity $\kappa : \R^+ \to \R^+$ which satisfies the following condition:

\begin{equation}  \label{eq:conditionkappa1} 
\exists A > 0, \, \, \, \limsup_{t \to 0^+}  \left(\frac{\kappa (A t)}{A \kappa (t)}\right) < \frac{1}{2 n}\cdot 
\end{equation}
The main result of this note is the following.

\smallskip
\smallskip

{\bf Main Theorem. }
{\it Let  $u : {\Omega} \longrightarrow \R$ be a  bounded  subharmonic function on $\Omega$ and $\kappa$ a modulus of continuity satisfying the condition (\ref{eq:conditionkappa}).  Assume that  there exists a constant $C_0 > 0$ and $0< \delta_1 < \delta_0$ such that for $0< \delta < \delta_1$,
\begin{equation} \label{eq:meanvalue}
\Lambda_\delta u (x)  - u (x) \,  \leq \, C_0  \, \kappa(\delta), \, \, \, \text{for} \, \, \, \, x \in  \Omega_\delta.
\end{equation}

Then there  exists constants $B > 1$,  $ \varepsilon_0 > 0$ with $B  \varepsilon_0  < \delta_1$ and a constant $C > 0$ such that for $0< \delta < \varepsilon_0$ and $x \in \Omega_{B\delta}$, we have
 \begin{equation} \label{eq:maximum}
\mathcal{M}_\delta u(x)   - u (x) \, \, \leq  \, \, C  \, \kappa(\delta). 
\end{equation}
In particular $u$ is $\kappa$-continuous on any compact subset $E \Subset \Omega$.}

\smallskip

Observe that the condition of the theorem is satisfied for any harmonic function, regardless on its regularity at the boundary. Hence we cannot expect to conclude anything about the full modulus of continuity of $u$ on the whole domain $\Omega$.

However if we know about the behaviour of $u$ near the boundary we can get a better control on the modulus of continuity of $u$ on $\bar{\Omega}$.

\smallskip
\smallskip

{\bf Corollary. }
{\it Let  $u : {\Omega} \longrightarrow \R$ be a  bounded  subharmonic function on $\Omega$ and $\kappa$ a modulus of continuity satisfying the condition (\ref{eq:conditionkappa1}).  

Assume that $u$ satisfies  the condition (\ref{eq:meanvalue}) and  extends as a $\kappa$-continuous function near the boundary.  Then $u$ is $\kappa$-continuous on $\bar{\Omega}$ i.e. 
\begin{equation} \label{eq:MCestimate}
\vert u (x) - u (y) \vert \leq L  \, \kappa (\vert x - y \vert),
\end{equation}
for any $x, y \in \bar{\Omega}$, where $L > 0$ is a uniform constant.}

 Let us recall the definition of  $\kappa$-continuity near the boundary. Set for $\delta > 0$,  
 $$
 \tilde{\kappa}_u (\delta) := \sup \{\vert u (x) - u (y)\vert ;  y  \in \partial \Omega, x \in \Omega \cap B (y,\delta)\}\cdot
 $$
 It is easy to see that if $u$ is continuous on $\partial \Omega$, then $\lim_{\delta \to 0}   \tilde{\kappa}_u (\delta) = 0$. 
 Then we say that $u : \bar{\Omega} \longrightarrow \R$ is $\kappa$-continuous near the boundary if there exists a constant $C > 0$ such that $\tilde{\kappa}_u (\delta) \leq C \kappa (\delta)$ for $\delta >0$ small enough.

\section{Characteristics associated to a subharmonic functions}
Let $\Omega \subset \R^n$ be a  domain. We denote by $\mathcal{SH} (\Omega) \subset  L^{1}_{loc}  (\Omega)$ the set of subharmonic functions on $\Omega$. 
We will will first introduce differents characteristics associated to $u$ and compare them.  Then we prove some average estimates on them.
\subsection{Basic definitions}
Let  $u\in \mathcal{SH} (\Omega)$. We  will associate to  $u$ the following characteristics.
Set $ \Omega_{\delta} := \{ x \in \Omega ; \mathrm{dist} (x,\partial \Omega) > \delta\}$  for $0 < \delta \leq  \delta_0$, where   $ \delta_0 >0$ is fixed so that $\Omega_{\delta_0} \neq \emptyset$.

Fix $0 < \delta < \delta_0$ and  $x \in \Omega_{\delta}$. Define the max characteristic of $u$ as follows 
\begin{equation} \label{eq:max}
\mathcal{M}_\delta u (x) := \max_{\bar B (x,\delta)} u = \max_{\vert \xi \vert = 1} u(x + \delta \xi),
\end{equation}
where $\bar{B} (x,r) := \{y \in \R^n ; \vert y - x\vert \leq r \}$ is the closed euclidean ball of center $x$ and radius $r>0$.

We define the mean value volume characteristic of $u$ 
\begin{equation} \label{eq:meanvol}
\Lambda_{\delta} u (x)  :=  \frac{1}{\tau_{n} r^n} \int_{B(x,\delta)} u ( y) d \lambda_n (y) = \frac{1}{\tau_{n}} \int_{\B} u (x + \delta y) d \lambda_n (y),
\end{equation}
where $d \lambda_n$ is the Lebesgue measure on $\R^n$ and $\tau_n := \lambda_n (\B)$ is the volume of the unit ball $\B \subset \R^n$.

We define the mean value area characteristic of $u$ 
\begin{equation} \label{eq:meanarea}
\mathcal{A}_\delta u (x)  := \frac{1}{\sigma_{n-1} } \int_{\mathbb S} u (x + \delta \xi) d \sigma (\xi),
\end{equation}
where $ d \sigma$ is the area measure of the unit sphere $\mathbb S = \partial \B$ and $\sigma_{n-1}$ is the area of the unit sphere $\mathbb S \subset \R^n$.

We consider more general mean value characteristics associated to the subharmonic function $u$.

 Let $\rho $ be a radial bounded Borel function with compact support in the unit ball $\B \subset \R^N$ such that $\int_\B \rho (x) d \lambda_n (x) = 1$. This means in spherical coordinates that
 \begin{equation} \label{eq:normalization}
 \sigma_{n-1}.  \int_0^1 \rho (r) r^{n -1} d r = 1.
 \end{equation}
 
 For $\delta >0$ and $x \in \R^n$,  we set
$\rho_\delta (x) := \delta^{-n} \rho(x\slash \delta)$.
Then it's well known that $\rho_\delta\to \epsilon_0, $ in the sense of distributions on $\R^n$ as $\delta\to 0$, where $\epsilon_0$ is the unit mass Dirac distribution at the origin. 
  
For $x \in \Omega_\delta$, the smooth mean value characteristic of $u$ is defined by

\begin{equation} \label{eq:meanvolsmooth}
\mathcal{R}_\delta u (x)= u \star \rho_\delta(x) := \int_{\B (x,\delta)} u (\xi) \rho_\delta (x - \xi) d \lambda_n (\xi) = \int_{\B} u (x + \delta \xi) \rho  (\xi) d \lambda_n (\xi) .
 \end{equation}
Observe that if $\rho =  \frac{1}{\tau_n} {\bf 1}_{\B}$ then $\mathcal{R}_\delta u (x) = \Lambda u_{\delta} (x)$.

It is well known that all these functions are non decreasing in $\delta$ and convex in the variable $t = k_n (r)$, where $k_2 (r) = \log r$ and $k_n (r) := - r^{2 - n}$ when $n \geq 3$ (see \cite{AG01}). 
 
 \subsection{Comparison of characteristics}
We want to compare all these characteristics for a subharmonic function.

\begin{lemma} \label{lem:comparison1}  Let  $\Omega \subset \R^n$ be an open set and  $u \in  \mathcal{SH} (\Omega)$. Then for any $0< \delta < \delta_0$ and $x \in \Omega_\delta$, we have 
$$
 b_n \left(\mathcal{A}_{\delta \slash 2} u (x) - u (x)\right) \leq  \mathcal{R}_{\delta} u  (x) - u (x)  \leq \mathcal{A}_\delta u  (x) - u (x), 
$$
In particular,
$$
 b_n \left(\mathcal{A}_{\delta \slash 2} u (x) - u (x)\right) \leq  \Lambda u_{\delta}  (x) - u (x)   \leq \mathcal{A}_\delta u (x) - u (x) \leq \mathcal{M}_\delta u (x) - u (x),
$$
where $b_n := \int_{1 \slash 2}^1 \rho  (r) r^{n-1} d r < 1 \slash \sigma_{n-1}$.
\end{lemma}
\begin{proof}
Fix $0 < \delta < \delta_0$ and $x  \in \Omega_\delta$.  Using spherical coordinates, we see that 
\begin{eqnarray*}
\mathcal R_\delta u (x)  & = & \int_0^1 \rho (r) r^{n-1} \, \left (\int_{\mathbb S} u (x + r\delta\xi) d \sigma (\xi)\right) d r \\
& = & \sigma_{n-1} \int_0^1 \rho (r) r^{n-1} \, \mathcal{ A}_{r \delta} u (x) d r.
\end{eqnarray*}
Therefore by the equation (\ref{eq:normalization}), we deduce that for any fixed $x \in \Omega_\delta$
\begin{equation} \label{eq:intmean}
\mathcal R_\delta u (x) - u (x) = \sigma_{n-1} \int_0^1 \rho (r) r^{n-1} \left(\mathcal{A}_{r \delta} (x) - u (x)\right) d r.
\end{equation}
Since the function $]0,\delta] \ni s \mapsto \mathcal{A}_{s} u (x) - u (x)$ in  non-negative and non-decreasing, it follows that 

\begin{equation} \label{eq:fineq}
b_n (\mathcal{A}_{\delta\slash 2} u  (x) - u (x)) \leq \mathcal{R}_\delta u (x) - u (x)  \leq \mathcal{A}_{\delta} u (x) - u (x),
\end{equation}
where $b_n := \int_{1 \slash 2}^1 \rho (r) r^{N-1} d r < 1$.

Now applying the formula (\ref{eq:intmean}) with $\rho = \frac{1}{\tau_n} {\bf 1}_{\B}$, we obtain 
$$
 \Lambda _{\delta \slash 2} u  (x) - u (x) = n  \int_0^1 r^{n-1} (\mathcal{A}_{r {\delta \slash 2} } u (x) - u (x)) d r,
$$
since $\sigma_{n-1} = n \tau_n$.

Applying the inequality (\ref{eq:fineq}) we obtain for $0 < \delta < \delta_0$ and $x \in \Omega_{\delta}$
$$
\Lambda_{\delta \slash 2} u (x) - u (x)  \leq \mathcal{A}_{\delta \slash 2} u  (x) - u (x) \leq b_n^{-1} (\mathcal{R}_{\delta} u (x) - u (x)).
$$
\end{proof}

Now we want to compare the supnorm and the mean value of a subharmonic function on balls.

\begin{lemma}  \label{lem:comparison2}  There exists $\delta_0 > 0$  small enough and a constant $a_n > 0$ such that  for any $ 0 <  \delta < \delta_0$, $0 <  \theta < 1$ and  $x \in \Omega_{ \delta}$,
we have

\begin{eqnarray}
\mathcal{M}_{\theta \delta} u (x) - u (x)   &\leq &  c_n \left(\mathcal{A}_{\delta} u (x) - u (x)\right)  \nonumber \\
  &+&  \frac{  c_n 2^n  \theta }{(1  - \theta)^{n - 1}}  \int_{\mathbb S} \left(u (x + \delta y)  - u (x)\right)_+  \, d \sigma (y) \cdot  
\end{eqnarray}
In particular
\begin{eqnarray}
\mathcal{M}_{\delta \slash 2} u (x) - u (x)  & \leq & c_n \left(\mathcal{A}_{\delta} u (x) - u (x)\right)  \nonumber\\
& +&   4^n c_n \int_{\mathbb S} \left(u (x + \delta y)  - u (x)\right)_+  \, d \sigma (y).
\end{eqnarray}

\end{lemma}

\begin{proof}
Assume first that $u$ is subharmonic in a neighbourhood of the closed ball $\bar{\B}$. It follows from Poisson-Jensen formula for the unit ball $ \B$ (see \cite{AG01}) that
\begin{equation} \label{eq:PJ-Ineq}
u (x) \leq \int_{\mathbb S} P (x,y) u (y) d \sigma (y), \, \, \, x \in \B,
\end{equation}
where
$$
P (x,y) := c_n \frac{1- \vert x\vert^2}{\vert x - y\vert^n}, (x,y)\in \B \times \partial \B,
$$
is the Poisson kernel of the unit ball $\B \subset \R^n$.

Since $\int_{\mathbb S} P (x,y) d \sigma (y) = 1$, it follows from (\ref{eq:PJ-Ineq}) that for any $x \in \B$,
\begin{equation} \label{eq:Ineg1}
u (x) - u (0) \leq \int_{\mathbb S} P (x,y) (u (y) - u (0)) d \sigma (y)  = I^+ (u) + I^- (u),
\end{equation}
where
\begin{equation} \label{eq:Iplus}
I^+(u) := \int_{\{y \in \mathbb S; u (y) \geq u (0)\}} P (x,y)  (u (y)  - u (0)) \, d \sigma (y), 
\end{equation}
and 
\begin{equation} \label{eq:Imoins}
I^-(u) :=  \int_{\{y \in \mathbb S; u (y) \leq u (0)\}} P (x,y)  ( {u (y)  - u (0)}) \, d \sigma (y).
\end{equation}
Since for $0< r < 1$ and $\vert x \vert = r$, we have
\begin{equation} \label{eq:PoissonIneq}
c_n \frac{1-r}{(1 + r)^{n-1}} \leq P (x,y) \leq c_n \frac{1+r}{(1 - r)^{n-1}}, 
\end{equation}
it follows from (\ref{eq:Iplus}) and (\ref{eq:PoissonIneq})  that for $\vert x\vert = r < 1$,
\begin{equation} \label{eq:I+Est}
I^+ (u) \leq  c_n \frac{1+r}{(1 - r)^{n-1}}  \int_{\{y \in \mathbb S; u (y) \geq u (0)\}}  (u (y)  - u (0)) \, d \sigma (y).
\end{equation}
Moreover, it follows from (\ref{eq:Imoins}) and (\ref{eq:PoissonIneq})  that for $\vert x\vert = r < 1$,
\begin{equation} \label{eq:I-Est}
I^-(u) \leq  c_n \frac{1-r}{(1 + r)^{n-1}}  \int_{\{y \in \mathbb S; u (y) \leq u (0)\}}  (u (y)  - u (0)) \, d \sigma (y).
\end{equation}
Therefore from (\ref{eq:I+Est}),  (\ref{eq:I-Est}) and (\ref{eq:Ineg1}), we deduce that
\begin{eqnarray*}
u (x) - u (0) &\leq & c_n \frac{1+r}{(1 - r)^{n-1}} \int_{\{y \in \mathbb S; u (y) \geq u (0)\}}  (u (y)  - u (0)) \, d \sigma (y) \\
& +  & c_n \frac{1-r}{(1 + r)^{n-1}} \int_{\{y \in \mathbb S; u (y) \leq u (0)\}}  (u (y)  - u (0)) \, d \sigma (y).
\end{eqnarray*}
Observe that for $r < 1$,
$$
\frac{1+r}{(1 - r)^{n-1}}  \leq  1 + \frac{2^n r}{(1 - r)^{n-1}} , \, \, \text{and}  \, \, \, \frac{1-r}{(1 + r)^{n-1}} \geq 1 - n r.
$$

This implies that
\begin{eqnarray*}
\max_{\vert x \vert = r} u (x) - u (0)  &\leq  & c_n  \int_{\mathbb S}  (u (y)  - u (0))  \, d \sigma (y) \\
&+&  \frac{c_n 2^n r}{(1 - r)^{n-1}} \int_{\mathbb S} (u (y) - u (0))_+ \, d \sigma (y) \\
&-&   n c_n r  \int_{\mathbb S} (u (y) - u (0))_- \, d \sigma (y) .
\end{eqnarray*}
Now in the general case fix $\delta_0 > 0$ small enough so that  $\Omega_{2 \delta_0} \neq \emptyset$. We fix $x \in \Omega_{2 \delta}$ so that $B (x,2 \delta) \subset \Omega$ and apply the previous inequality to the function $y \longmapsto u (x + \delta y)$ which is subharmonic in a neighbourhood of the unit ball $\bar{\B}$. 
We then obtain for $0 < r < 1$
\begin{eqnarray*}
\max_{\B (x, r \delta)} u  - u (x) &\leq  & c_n  \int_{\mathbb S}  (u (x + \delta y)  - u (x))  \, d \sigma (y) \\
& +& \frac{c_n 2^n r}{(1 - r)^{n-1}} \int_{\mathbb S} (u (x + \delta y)  - u (x))_+ \, d \sigma (y) \\
&-&  n c_n r  \int_{\mathbb S} (u (x + \delta y)  - u (x))_- \, d \sigma (y).
\end{eqnarray*}
Since the last term is non positive, the inequality of the lemma follows.
\end{proof}
We can easily deduce the following result.
\begin{corollary} Let $u $ be a bounded  subharmonic function on a bounded domain $\Omega \subset \R^n$. Assume that there exists $\alpha \in ]0,1]$ and $\kappa_1 > 0$ such that for any $ 0 < \delta < 2 \delta_0$ and $x \in \Omega_{\delta}$, we have 
$$
\mathcal{R}_\delta u (x) \leq u (x) + \kappa_1 \delta^{\alpha}.
$$
Then there exists $\kappa_2 > 0$ and $0< \delta_1 <<1$ such that for any $\delta < \delta_1 $ and $x \in \Omega_\delta$, we have
$$
\mathcal M_{\delta} u (x) \leq u (x) + \kappa_2 \delta^{\alpha \slash (1 + \alpha)}.
$$
\end{corollary}

\begin{proof} Apply Lemma \ref{lem:comparison2} with $\theta = \delta^\alpha$. Then for $\delta < \delta_1 <  2^{- 1 \slash \alpha}$ and $x \in \Omega_{\delta^{1 + \alpha}}$, we have
$$
\mathcal{M}_{\delta^{1 + \alpha}} u (x) - u (x)   \leq  c_n \kappa_1 \delta^\alpha +    L_n  M \delta^{\alpha}.
$$
where $M := \text{osc}_{\bar \Omega} u$ is the oscillation of $u$ on $\bar \Omega$ and $L_n > 0$ is a uniform constant.
Relpacing $\delta$ by $\delta^{1 \slash (1 + \alpha)}$ we obtain that $u$ is H\"older continuous with exponent $\alpha \slash (1 + \alpha)$ and $\kappa_2 := a_n \kappa_1 + n M L_n$.
\end{proof}

\subsection{Average estimates } 
 We first recall a well know result which is important in applications. This result is shown in \cite{CKZ08}, beut we will recall the proof here for the convenince of the reader.
 \begin{lemma}
 Let $u$ be subharmonic function on  a bounded domaine $\Omega \subset \R^n$. Then there exists a uniform constant $e_n > 0$ such that for $0 < \delta < \delta_0$,
 $$
\int_{\Omega_ \delta} \left(\mathcal{A}_\delta u (x) - u (x)\right) d  y \leq e_n \Vert \Delta u\Vert_{\Omega_\delta} \, \delta^2.
 $$
 
 \end{lemma}
 \begin{proof}
Let $\mu := (1\slash 2 \pi)  \Delta u $ be  the Riesz measure of $u$ on $\Omega$.   It follows from Poisson- Jensen formula that for $x \in \Omega$ with $u(x) >  - \infty$, we have
 $$
 \mathcal{A}_\delta u (x) - u (x) = \int_0^\delta t^{1 - n} ( \mu  (B (x,t)) d t.
 $$
 Then integrating on $x$ over $\Omega$ and applying Fubini's Theorem we obtain
 \begin{eqnarray*}
 \mathcal{A}_\delta u (x) - u (x) d \lambda_n (x) &\leq &\int_0^\delta t^{1 - n} \int_{\Omega_\delta}  d \mu (\zeta) \int_{(B(\zeta,t) } d \lambda_n (z)  d t \\
 &\leq & \tau_n \mu (\Omega_\delta) \int_0^\delta \,  t \, d t =  \tau_n \mu (\Omega_\delta)  \delta^2\slash 2,
 \end{eqnarray*}
 where $\tau_n = \lambda_n (\B)$ is the volume of the unit ball in $\R^n$. This proves the Lemma.
 \end{proof}
\begin{corollary} Assume that $u $ is a  subharmonic function on a bounded domain $\Omega \subset \R^n$. Then for any $ 0 < \delta < \delta_0 \slash 2$, we have 
$$
\int_{\Omega_{2 \delta}} (\mathcal M_{\delta } u - u) d \lambda_n \leq p_n  \int_{\Omega_{2 \delta}} (\mathcal A_{2 \delta} u - u) d \lambda_n +   q_n \delta \Vert \nabla u\Vert_{L^1 (\Omega_{\delta})},
$$
where $p_n, q_n > 0$ are uniform constants.

In particular 
$$
\int_{\Omega_{2 \delta}} (\mathcal M_{\delta } u - u) d \lambda_n \leq p_n \Vert \Delta u\Vert_{\Omega_{2 \delta}} \, \delta^2 + q_n  \Vert \nabla u\Vert_{L^1 (\Omega_{\delta})} \, \delta,
$$
where $p_n, q_n > 0$ are uniform constants.
\end{corollary}
\begin{proof} We apply Lemma \ref{lem:comparison2} and integrate over $\Omega_{2\delta}$. Then for $0 < \delta < 2 \delta_0$,
\begin{eqnarray*}
\int_{\Omega_{2\delta}} (\mathcal M_{\delta} u (x) - u (x)) d \lambda (x) &\leq & c_n b_n \int_{\Omega_{2\delta}} \left(\mathcal A_{2\delta}  u (x) - u (x)\right) \\
& + &  c_n 2^{2 n - 2}  \int_{\mathbb S} \left(\int_{\Omega_{2\delta}}\vert u (x + \delta y)  - u (x)\vert  d \lambda (x)\right) \, d \sigma (y),
\end{eqnarray*}
where $p_n = c_n b_n$ and $q_ n :=  c_n 2^{2 n- 2}$.

We claim that for fixed $x \in \Omega_\delta$ and $0 < \delta < \delta_0/2$, we have
\begin{equation} \label{eq:mean-gradient}
\int_{\Omega_{2\delta}}\vert u (x + \delta y)  - u (x)\vert  d \lambda (x) \leq \delta \Vert \nabla u\Vert_{L^1 (\bar{\Omega}_{\delta})}.
\end{equation}

Indeed  assume first that $u$ is smooth. 
Now observe that for $\vert y \vert = 1$, we have
$$
u (x+ \delta y) - u (x) = \int_0^\delta D u (x+t y)\cdot y d t.
$$
Then using Fubini's Theorem, we obtain
$$
\int_{\mathbb S} \vert u (x+ \delta y) - u (x)\vert  \leq  \int_0^\delta \int_{\mathbb S}  \vert Du (x+t y)\cdot y\vert d \sigma (y)  d t.
$$
Integration over $\Omega_{2 \delta}$ leads to the inequality (\ref{eq:mean-gradient}). 

For a non smooth function $u$, we can approximate $u$ by a decreasing sequence $(u^j)$ of smooth subharmonic functions on a neighbourhood $V$ of $\bar{\Omega}_{\delta}$ so that  $D u^j \to D u$ in $L^1(V)$ and almost everywhere on $V$. Applying the inequality (\ref{eq:mean-gradient}) to the $u^j$'s  and passing to the limit, we obtain by Fatou's lemma the inequality  (\ref{eq:mean-gradient}) for $u$.
\end{proof}

\section{More general moduli of continuity}

We want to show that the sup-regualrization and the mean value regularisation have the same behaviour for a large class of moduli of continuity. 
Namely we will prove an important result which confirms a lemma  stated in  \cite{GKZ08}  and used in the litterature 
for a H\"older modulus of continuity. Chinh H. Lu discovered  recently a gap in the proof of \cite{GKZ08} which was fixed in \cite{LPT20} in the case of a compact hermitian manifold (without boundary) following the same scheme.

We will follow  the same scheme  as in \cite{GKZ08} and use  a new idea of \cite{LPT20}  to prove a more general result.

\subsection{A new characterization}
Let us first give some definitions. 

Let  $\kappa : [0,l] \longrightarrow \R^+$ be a modulus of continuity i.e.  a  continuous increasing subadditive function such that $\kappa (0) = 0$.

We will consider the following growth condition on $\kappa$. 
\begin{equation} \label{eq:conditionkappa}
\exists A >0, \, \, \,  \limsup_{t \to 0^+}  \left(\frac{2 n \kappa(A t}{ A \kappa(t)} \right)  \, < \, 1.
\end{equation}

Observe that this condition holds for a logarithmic H\"older modulus of continuity  defined by 
$$
\kappa_{\alpha,\beta} (t) = t^\alpha (-\log t)^{\beta}, \, \, 0 < t < t_0 < 1
$$  
with $0 \leq \alpha < 1$ and $\beta \in \R$, with $\beta < 0$ if $\alpha = 0$ and $t_0 >0$ is chosen so small that $\kappa_{\alpha,\beta} $ is concave on $[0,t_0]$.
However it's not satisfied by the modulus of continuity $\kappa_{1,\beta}$ with $\beta \leq  0$.

 We need another definition.
\begin{definition}
We say that a  function $u : \bar{\Omega} \longrightarrow \R$ is $\kappa$-continuous near  the boundary $\partial \Omega$ if there exits $\varepsilon_0 >0$ small enough such that  for any $\zeta \in \partial \Omega$ and $z \in \Omega$ with $\vert z - \zeta \vert \leq \varepsilon_0$, we have 
\begin{equation} \label{eq:contnearbd}
\vert u(z) - u(\zeta) \vert \leq \kappa(\vert z - \zeta\vert).
\end{equation}
\end{definition}
Observe that this condition implies the continuity of $u$ on $\partial \Omega$ and it is satisfied if there exists two functions $v,w$ $\kappa$-continuous near the boundary such that $v \leq u \leq w$ near the boundary  and $v=u=w$ on $\partial \Omega$.

We need to introduce one more characteristic  associated to $u$. For $0 < \delta < \delta_0$, we set
$$
\mathcal O_{\delta} u (x) := \text{osc}_{B (x,\delta)} u = \max \{ \vert u (y_1) - u (y_2) \vert ; y_1, y_2 \in B(x,\delta)\}\cdot
$$
    Now we can sate the main result of this note.
\begin{theorem} \label{thm:MCcontrol}
 Let  $\kappa$ be a modulus of continuity satisfying the condition (\ref{eq:conditionkappa}) and $ u :  \bar{\Omega} \longrightarrow \R$ be a bounded function which is  subharmonic on $\Omega$ and $\kappa$-continuous near  the boundary $\partial \Omega$.  
 
 Then the following properties are equivalent :

$(1)$ there exists a constant $L_1 > 0$ and $0< \delta_1 < \delta_0$ such that for $0< \delta < \delta_1$,
 $$
\mathcal M_\delta  u (x)   - u (x) \leq  L_1 \, \kappa(\delta), \, \, \, \text{for} \, \, \, \, x \in \Omega_\delta, 
$$ 

$(2)$ there exists constants $ L_2 > 0$ and $0< \delta_2 < \delta_0$ such that $B \delta_2 < \delta_0$ and for $0< \delta < \delta_2$,
$$
\mathcal{R}_\delta u(x)   - u (x) \leq L_2 \, \kappa(\delta), \, \, \, \text{for} \, \, \, \,x \in \Omega_{\delta}, 
$$

$(3)$ there exists a constants  $B > 1$, $L_3 > 0$ and $0< \delta_3 < \delta_0$ such that for $0< \delta < \delta_3$,
$$
\mathcal{O}_\delta u(x)  \leq L_3 \, \kappa(\delta), \, \, \, \text{for} \, \, \, \,x \in \Omega_{B \delta}, 
$$

 $(4)$  the function $u$ is $\kappa$-continuous on $\bar{\Omega}$ i.e. there exists a constant $L_3 >0$ and $0 < \delta_3 < \delta_ 0$ such that for any $x \in \bar{\Omega}$ and $y \in \bar{\Omega}$ with $\vert x - y \vert \leq \delta_3$, we have 
 $$
  \vert u(x) - u(y) \vert \leq   L_3 \, \kappa (\vert x -y\vert).
 $$
\end{theorem} 
Observe that the condition $(2)$ is always satisfied for any harmonic function $u$ on $\Omega$,
regardless of its behaviour at the boundary, while the  condition $(4)$ implies that the boundary values of $u$ is $\kappa$-continuous on $\partial \Omega$.
Therefore $(2)$ and $(4)$ are not equvalent without any condition on the behaviour of $u$ at the boundary.

The main step in the proof of our theorem is the following lemma whose proof is inspired from \cite{GKZ08} and \cite{LPT20}.
\begin{lemma} \label{lem:funlemma}
Let  $u : {\Omega} \longrightarrow \R$ be a  bounded  subharmonic function on $\Omega$ and $\kappa$ a modulus of continuity satisfying the condition (\ref{eq:conditionkappa}) .  Assume that  there exists a constant $C_0 > 0$ and $0< \delta_1 < \delta_0$ such that for $0< \delta < \delta_1$,
\begin{equation} \label{eq:convolution}
\mathcal{R}_\delta u (x)  - u (x) \,  \leq \, C_0  \, \kappa(\delta), \, \, \, \text{for} \, \, \, \, x \in  \Omega_\delta.
\end{equation}

Then there  exists constants $B > 1$,  $ \varepsilon_0 > 0$ with $B  \varepsilon_0  < \delta_1$ and a constant $C_3 > 0$ such that for $0< \delta < \varepsilon_0$ and $x \in \Omega_{B\delta}$, we have
 \begin{equation} \label{eq:maximum}
\mathcal{O}_\delta u(x)   \, \, \leq  \, \, C_3  \, \kappa(\delta). 
\end{equation}

In particular $u$ is $\kappa$-continuous on any compact set $E \Subset \Omega$.
\end{lemma}

\begin{proof}
We claim  that the condition (\ref{eq:conditionkappa}) implies that  we can choose  $A > 2$ large enough and $\delta_2 >0$ small enough  such that $(2 A + 1) \delta_2 < \delta_1$ and
\begin{equation} \label{eq:nu}
\theta := 2  n \sup_{0< t \leq \delta_2} \frac{\kappa( (A + 1) t)}{ A \kappa(t)}   \, < 1.
\end{equation}
Indeed  by (\ref{eq:conditionkappa}) we see that  there exists   $A' > 0$, $\nu < 1 \slash 2n$ and $0< \delta'_2 < \delta_0$ small enough   such that 
$$
\sup_{0< s \leq \delta'_2} \frac{\kappa(A'  s)}{ A' \kappa(s)} < \nu  < 1 \slash 2 n.
$$ 
Fix an integer $N > 1$ and apply this inequality for $s = N t$.  Then by subadditivity of $\kappa$, we have  for $0 < t < \delta_2 := \delta'_2 \slash N$, 
$$
  \frac{\kappa( (N A' t )}{ N A' \kappa( t)}  \leq  \frac{\kappa(N A' t )}{ A' \kappa( N t)}  = \frac{\kappa(A' s )}{ A'\kappa( s)}  < \nu.
$$
Now choose $N > 1$ so large that $N A' > 3$ and set $A := N A' - 1 > 2$. Then the previous inequality implies that for $0 < t < \delta_2$, 
$$
\frac{\kappa( ((A +1) t )}{ A \kappa( t)}   =  \frac{\kappa( (N A' t )}{ N A' \kappa( t)} \frac{N A'}{A} <  \nu \frac{N A'}{A}.  
$$
 Since $\frac{N A'}{A} =  \frac{N A'}{N  A' - 1} \to 1$  as $N  \to + \infty$ and $\nu < 1 \slash 2n$, we can find  $N > 2$ large enough so that $ \nu \frac{N A'}{A} < 1  \slash 2n$, which implies the inequality (\ref{eq:nu}) and proves the claim. 
 
Now choose $\delta_2 >0$ so small that   $(2 A + 1)\delta_2 < \delta_1$ and fix  $0 < \delta <  \delta_2$.

The first step of the proof follows the scheme given in \cite{GKZ08}.
Observe that the condition (\ref{eq:convolution}) implies that $u$ is continuous on $\Omega$. 

 Let $x_0 \in \Omega_{(2 A + 1)\delta}$. Then $\bar{B} (x_0,\delta \subset \Omega$ and  by continuity, there exists $\xi_0 , y_0\in \bar{B} (x_0,\delta)$ such that
\begin{equation}\label{eq2}
\mathcal{O}_\delta u  (x_0)  =u(y_0)-u(\xi_0).
\end{equation}
 
Since $\text{dist} (x_0, \partial \Omega) > (2 A + 1)\delta$, we have $\bar{\B} (\xi_0,2A\delta) \subset \Omega$.
 Set  $a := A - 2 > 0$  and $r := A \delta = (a + 2) \delta$. Then by Lemma \ref{lem:comparison1}, it follows that

\begin{equation} \label{eq:Ineq1}
\mathcal{R}_{2r} \,  u (\xi_0)-u(\xi_0) \geq b_n\left(\Lambda_{r} u (\xi_0) -u(\xi_0)\right),
\end{equation}
where $b_n := \sigma_{n-1} \int_{1/2}^1s^{n-1}\rho(s)ds$.

Since $\bar{\B}(y_0,a\delta)\subset \bar{\B}(\xi_0,  r) \subset \Omega$,  we can write,

\begin{eqnarray}
\frac{1}{\tau_{n}r^{n}}\int_{{\B}(\xi_0,r)}u(y)d\lambda_n (y) &=& \frac{1}{\tau_{n}r^{n}}\int_{\mathbb{B}(y_0,a\delta)}u(y)d\lambda_n (y) \nonumber \\
&+& \frac{1}{\tau_{n}r^{n}} \int_{\mathbb{B}(\xi_0,r)\setminus\mathbb{B}(y_0,a\delta)}u(y)d\lambda_n (y).
\end{eqnarray}
By subharmonicity of $u$, we have 
$$
\frac{1}{\tau_{n}r^{n}}\int_{\mathbb{B}(y_0,a\delta)} u(y)d\lambda_n (y) 
\geq \frac{a^{n}}{(a+2)^{n}}u(y_0).
$$
On the other hand, we have

\begin{eqnarray*}
\frac{1}{\tau_{n}r^{n}} \int_{\mathbb{B}(\xi_0,r)\setminus\mathbb{B}(y_0,a\delta)} u(y)d\lambda_n (y) 
 \geq    \left(1-\frac{a^{n}}{(a+2)^{n}}\right) \left(u (y_0) -  \mathcal{O}_ {r+\delta} \,  u (x_0)\right).
\end{eqnarray*}

Therefore adding the last two inequalties, we obtain 
\begin{eqnarray*}
\Lambda_{r} u (\xi_0) - u(\xi_0) &\geq & \frac{a^{n}}{(a+2)^{n}} u(y_0) - u(\xi_0)  \\
&+ & \left(1-\frac{a^{n}}{(a+2)^{n}}\right) \left(u (y_0) -  \mathcal{O}_{r+\delta} \,  u (x_0)\right)\\
& = & \mathcal{O}_\delta u (x_0)   - \left(1-\frac{a^{n}}{(a+2)^{n}}\right) \mathcal{O}_{r+\delta} \,  u (x_0).
\end{eqnarray*}
From (\ref{eq:Ineg1}) and the previous estimate, we deduce that 
$$\mathcal{R}_{2 r} u  (\xi_0) - u(\xi_0)  \geq b_n \left(\mathcal{O}_\delta u (x_0)   - a_n  \mathcal{O}_{r + \delta} u (x_0)\right),
$$
where $a_n := (1-\frac{a^{n}}{(a+2)^{n}}$.

The second step of the proof will use  an idea of \cite{LPT20}.

Set $ f (t) := \mathcal{O}_t u (x_0)$ and $g(t) := \mathcal{R}_{2 t} u (\xi_0) - u(\xi_0)$ for $0 < t < \delta_2$ and  observe that $a_n \leq \frac{2 n}{A}$, with $A := a + 2$.
Then  for $0 < \delta < \delta_2$, we have 

$$
f(\delta) \leq (2 n/A)   f((A + 1) \delta) + (1\slash b_n) g (A \delta).
$$
Now since $\xi_0 \in \Omega_{2 A \delta}$ and $2 A\delta <  \delta_1$, it follows from (\ref{eq:convolution}) that $g (A \delta) \leq C_0 \kappa (2 A r)$. Then by subadditivity of $\kappa$, we have  for $0 < \delta < \delta_2$

$$
f(\delta) \leq  (2 n/A) f((A+1) \delta) + C_1 \kappa  (\delta),
$$
where $C_1 :=  (2 A +1) C_0\slash b_n$.

 Consider the quotient function $h(t) := f(t) \slash \kappa (t)$. From the previous estimate we deduce that  for $0 < \delta < \delta_2$,
$$
h (\delta) \leq  2 n   \frac{\kappa ((A+1)\delta)}{A \kappa (\delta)} h (B\delta)  + C_1,
$$
where $B := A +1$. Now as $A > 2$ is choosen so that the condition (\ref{eq:conditionkappa}) is satisfied, there exits  $0< \delta_3  < \delta_2$ so that for $0 < \delta <\delta_3$, we have  $ 2 n   \frac{\kappa ((A+1) \delta)}{A\kappa (\delta)} =: \theta < 1$. Then for $0 < \delta <\delta_3$ we have
$$
h (\delta) \leq  \theta  h (B\delta) + C_1,
$$
Iterating this inequality we see that for any $0 < \delta < \delta_3$ and any $k\in \N$, we have
\begin{equation} \label{eq:fundeq}
h (\delta B^{-k} ) \leq  \theta^k  h (\delta) + C_2,
\end{equation}
where $C_2 := C_1 \frac{1}{1-\theta}$.

We claim that  this inequality implies  that $ h (t) $ is bounded near $0$. 
Indeed choose $\varepsilon_0 > 0$ such that $B^2 \varepsilon_0  < \delta_3$ and set
$$
M_0 := \max \{ h (t) \, ; \, \varepsilon_0  \leq t \leq B^2 \varepsilon_0 \}.
$$
Fix $0 < t < \varepsilon_0 $ and choose an integer $k$ so that $t B^k \in [\varepsilon_0 ,  B^2 \varepsilon_0]$.  Then applying the inequality (\ref{eq:fundeq}) with $\delta = t B^k$ we obtain
$$
h (t) \leq \theta^k h (t B^k) + C_2 \leq M_0 + C_2 =: C_3,
$$
which proves our claim.

Let $E \Subset \Omega$ be a compact set and $r_0 := \mathrm{dist} (E, \partial \Omega) >0$. Then for $0 < \delta < r_0 \slash (2A + 1)$ , $E \subset \Omega_{(2A + 1)\delta}$. Then we can apply the previous estimate and get for $x \in E$ and $y \in E$ with $\vert x - y\vert \leq \delta$, $u (x) - u(y) \leq C_3 \kappa (\delta)$.
\end{proof}

We do not know if the condition (\ref{eq:convolution}) implies  that the condition (\ref{eq:maximum}) holds for any $x \in \Omega_\delta$, nor if it implies that $ u $ is $\kappa$-continuous on $\bar{\Omega}$.

\subsection{Proof of the main theorem}
We are now ready to prove the main theorem stated in the introduction.

\begin{proof}  

  It follows from  from the assuptions od the main theorem  that $u$ is continuous on $\overline{\Omega}$. Therefore   we can find $x_0\in \overline{\Omega}$, $\xi_0\in\overline{\Omega}$ such that  $|x_0-\xi_0|\leq\delta$ and
\begin{equation}\label{eq2}
\kappa_u (\delta)= \sup_{\xi, x \in \bar \Omega}  (u (\xi) - u(x)) =u(\xi_0)-u(x_0).
\end{equation}

Take $0 <  \delta_3$ small enough so that $(B + 1) \delta_3 < \delta_2$ and fix $0 < \delta < \delta_3$. Then there are two cases to be considered for the point $x_0$.

1)  If $x_0 \in \Omega$ and $\mathrm{dist} (x_0,\partial \Omega) > B \delta$, then $x_0 \in \Omega_{B \delta}$ and $\xi_0 \in \bar{\B} (x_0,\delta)$ and then  by the inequality  (\ref{eq:maximum}) we have $\kappa_u (\delta) = u (\xi_0)  - u (x_0) \leq L_3 \kappa(\delta).$

2)  If  $x_0 \in \bar{\Omega}$ and  $\mathrm{dist} (x_0,\partial \Omega) \leq  B\delta$,  we can choose $y_0 \in \partial \Omega$ such that $\vert y_0 - x_0\vert = \mathrm{dist} (x_0,\partial \Omega) \leq B \delta$.
 Then $\vert\xi_0 - y_0\vert \leq (B + 1)\delta \leq \varepsilon_0$ and $\vert y_0 - x_0\vert \leq \varepsilon_0$. Since $u$ is $\kappa$-continuous  near the boundary , taking $\delta_3$ small enough,  it follows that $\vert u(x_0) - u (y_0)\vert \leq C_0 \kappa (B \delta)$ and $\vert u(\xi_0) - u (y_0)\vert \leq C_0 \kappa ((B + 1)\delta)$.

This implies that
\begin{eqnarray*}
\kappa_u (\delta) = u (\xi_0) - u(x_0) &\leq & u (\xi_0) - u(y_0)  + u(y_0) - u(x_0)  \\
&\leq  & C_0 \kappa (B \delta) +  C_0 \kappa ((B + 1)\delta,
\end{eqnarray*}
 which by subadditivity implies
$\kappa_u (\delta) \leq 2(B + 2)  C_0 \kappa (\delta)$. This proves the theorem.
\end{proof}

\noindent{\bf Question :} Is the main theorem true for for any modulus of continuity?

\subsection{The case of quasi-plurisubharmonic functions}

Let $(X,\omega)$ be a compact Hermitian manifold of complex dimension $n$. Let $d$ be the geodesic  distance on $X$ associated to the metric $\omega$.  Let $\varphi$ be an $\omega$-plurisubharmonic function on $X$.  We  define the modulus of continuity of $\varphi$ as follows. For $\delta > 0$ set
\begin{equation} \label{eq:moduluscontinuity}
\kappa_\varphi (\delta) := \sup \{  \varphi (x) - \varphi (y) \, ; \, (x,y) \in X^2 , d(x,y) \leq \delta\}. 
\end{equation}

On the other hand,   we can define the local regularization $\mathcal{R}_\delta \varphi$ of $\varphi$ on a neighbourhood of each point $x_0$ using a local chart $(U,F)$ centered at $x_0$ as follows :  if $F : U \longrightarrow \C^n$ is a biholomorphism from a neighbourhood $U$  of $x_0$ to a bounded domain $\Omega \Subset \C^n$ such that $F (x_0) = 0$. Then we define for $x \in U_\delta := F^{-1}(\Omega_\delta)$,
$$
\mathcal{R}_\delta \varphi (x) := (\varphi \circ F^{-1})_\delta \circ F (x),
$$
where $\Omega_\delta := \{z \in \Omega \ ; \ \text{dist} (z, \partial \Omega) > \delta\}$ and $ (\varphi \circ F^{-1})_\delta$ is the standard  regularization of the quasi-psh function $   \varphi \circ F^{-1}$ on $\Omega$.

We consider a modulus of continuity  $\kappa : [0,l] \longrightarrow \R^+$ satisfying the following growth condition 
\begin{equation} \label{eq:conditionkappa2}
\exists A >0, \, \, \,  \limsup_{t \to 0^+}  \left(\frac{ 4 n \kappa(A t}{ A \kappa(t)} \right)  \, < \, 1.
\end{equation}

The following result was used in \cite{DDGKPZ14} for a H\"older modulus of continuity and was proved recently in \cite{LPT20} in that case. We will prove the following more general version using our previous results. 
\begin{theorem}
 Let  $\kappa$ be a modulus of continuity satisfying the condition (\ref{eq:conditionkappa2}) and $ \varphi :  X \longrightarrow \R$ be a bounded $\omega$-plurisubharmonic function on  $X$.  Assume that  for any point $x_0 \in X$ there exists a local chart $(U,z)$ contered at $x_0$ and  constants $C_1 > 0$ and $0< \delta_1 << 1$ such that for any $x \in U$ and any $0< \delta < \delta_1$,
\begin{equation} \label{eq:LocalReg}
\mathcal{R}_\delta \varphi (x)   - \varphi (x) \leq C_1 \kappa(\delta), 
\end{equation}
where $\mathcal{R}_\delta \varphi$ is a local regularization of $\varphi$ in the local chart $U$.

Then there exists a constant $C_2 > 0$  such that  for any $x \in X$ and $y \in X$, we have 
 $$
  \vert u(x) - u(y) \vert \leq   C_2 \kappa (d( x , y)).
 $$
\end{theorem} 
\begin{proof}  It follows  from our hypothesis that the function $u$ is continous on $X$.  Indeed fix an arbitrary point $x_0 \in X$ and localize the problem in a neighbourhood of $x_0$ so that the inequailty (\ref{eq:LocalReg}) is staisfied.  
Fix a neighbourhood $Y$  of $x_0$in $U$  and  a  biholomorphism $F : Y \longrightarrow F(Y) \subset \C^n$ from $Y$ to  a neighbourhood of the closed euclidean unit ball $\bar{\B} \subset \C^n$ so that the point $x_0$ is sent onto $0$.
Since $\omega > 0$ there exists a constant $C > 0$ such that $\beta \leq \omega \leq C  \beta $ on  $Y$, where $F_* (\beta))$ is a multiple of  the standard K\"ahler form on $\C^n$.  Since $u$ is $\omega$-plurisubharmonic, and $\omega \geq \beta$, $u$ is $\beta$-plurisubhramonic on $Y$ and we can choose a local smooth potential $w$ for $\beta $ on  $Y$ so that the function  $v  := \varphi + w$ is plurisubharmonic on $Y$. 
 
Since $d_\omega  (x,y) \sim d_\beta (x,y) $ on $Y$ and $\mathcal{R}_\delta v = \mathcal{R}_\delta \varphi + O (\delta)$ on $Y$ , we are then reduced to the case where $v$ is a plurisubharmonic function on a neighbourhood  of $\bar{\B}$ which satisfies $\mathcal{R}_\delta v - v \leq L_1 \kappa (\delta)$ on a neighbourhood of $\bar \B$ for some constant $L_1 > 0$. This implies that $v$ is continuous on $\B$, hence $u$ is continuous on a neighbourrhood of $x_0$. 
This proves the continuity of $\varphi$ on $X$ since $x_0$ was arbitrary.

In remains to prove the estimate on the modulus of continuity of $\varphi$.  By continuity of $u$  there exists 
$(x_\delta, y_\delta) \in X^2$ such that $d (x_\delta,y_\delta) \leq \delta$ and
$$
\kappa_\varphi (\delta) = \varphi (x_\delta) -  \varphi  (y_\delta).
$$ 
We want to show that there exists a constant $C_2 > 0$ such that 
$$
\limsup_{\delta \to 0^+}  \frac{\kappa_\varphi (\delta)}{\kappa (\delta)} \leq  C_2.
$$

It's enough to show that the limsup along any sequence $\delta_j \to 0$ is uniformly bounded by the same constant $C_2$. Up to extracting a subsequence, we can assume  by compactness  that there exists a point $x_0 \in X$  such that  $(x_{\delta_j},y_{\delta_j}) \to (x_0,x_0)$ as $j \to + \infty$. 
 
 Since  $\mathcal{R}_\delta v - v= \mathcal{R}_\delta \varphi - \varphi + O (\delta) $, applying the previous 
localization process  at the point $x_0$, we are reduced to the case where $v$ is a plurisubharmonic function on a neighbourhood $\Omega$ of $\bar{\B}$ which satisfies $\mathcal{R}_\delta v - v \leq C_1 \kappa (\delta)$ on a neighbourhood of $\bar \B$.  Then we can assume that  all the points  $x_{\delta_j}$ and $y_{\delta_j}$ belong to $\B$ for  any $j\in \N$.

From Lemma \ref{lem:funlemma} there exists a constant $L_2 > 0$ such that  for $j >  1$ large enough,  we have $v (x_{\delta_j}) - v(y_{\delta_j}) \leq L_2 \kappa (\delta_j)$, which implies  that   $\kappa_\varphi (\delta_j) = \varphi (x_{\delta_j}) - \varphi (y_{\delta_j}) \leq C_2 \kappa (\delta_j)$.  The theorem is proved.
\end{proof}

\smallskip

\smallskip

 {\bf Aknowledgements:} We thank Chinh H. Lu for a careful reading of the preliminary version of this note and for useful comments on the proof of Lemma 3.3. We also thank Vincent Guedj and Henri Guenancia for useful discussions on the subject of this note.

\end{document}